\documentclass[oneside, 12pt]{amsart}

\usepackage{amsthm}
\usepackage[T1]{fontenc}
\usepackage{ifpdf}

\def\NoLabels{YES}
\def\HYPER{YES}

\ifpdf\relax\else\let\HYPER=\undefined\fi

\ifdefined\HYPER
\usepackage[hidelinks, pdfstartview=FitB, pdfpagemode=UseNone]{hyperref}
\else
\newcommand{\href}[2]{\relax}
\newcommand{\url}[1]{#1}
  \def\NoLabels{YES}
\fi

\ifdefined\NoLabels\relax\else%
\usepackage[condensed]{tgheros}

\usepackage[notref, notcite]{showkeys}

\fi
\DeclareRobustCommand{\SkipTocEntry}[5]{}

\usepackage{amssymb,nccmath}
\usepackage{esint}
\usepackage{mathtools}
\usepackage[english]{babel}
\usepackage{epsfig}
\usepackage{mathptmx}
\usepackage{amsmath,amsfonts,amssymb}
\usepackage{mathtools}
\usepackage{enumitem}
\usepackage{mathrsfs}
\usepackage[all]{xy}
\usepackage{graphicx}
\usepackage{latexsym}
\usepackage{verbatim}
\usepackage{xcolor}
\usepackage{xifthen}

\makeatletter
\newsavebox\myboxA
\newsavebox\myboxB
\newlength\mylenA

\newcommand*\xoverline[2][0.75]{%
    \sbox{\myboxA}{$\m@th#2$}%
    \setbox\myboxB\null% Phantom box
    \ht\myboxB=\ht\myboxA%
    \dp\myboxB=\dp\myboxA%
    \wd\myboxB=#1\wd\myboxA% Scale phantom
    \sbox\myboxB{$\m@th\overline{\copy\myboxB}$}%  Overlined phantom
    \setlength\mylenA{\the\wd\myboxA}%   calc width diff
    \addtolength\mylenA{-\the\wd\myboxB}%
    \ifdim\wd\myboxB<\wd\myboxA%
       \rlap{\hskip 0.5\mylenA\usebox\myboxB}{\usebox\myboxA}%
    \else
        \hskip -0.5\mylenA\rlap{\usebox\myboxA}{\hskip 0.5\mylenA\usebox\myboxB}%
    \fi}
\makeatother

\usepackage[normalem]{ulem}
\usepackage{accents}

\usepackage{marginnote}
\reversemarginpar
\usepackage{cancel}

\usepackage{mparhack}
\usepackage[textwidth=2.5cm]{todonotes}
\newtheorem{theorem}{Theorem}[section]
\newtheorem{lemma}[theorem]{Lemma}
\newtheorem{proposition}[theorem]{Proposition}
\newtheorem{corollary}[theorem]{Corollary}

\theoremstyle{definition}
\newtheorem{definition}[theorem]{Definition}
\newtheorem{example}[theorem]{Example}

\theoremstyle{remark}
\newtheorem{remark}[theorem]{Remark}

\numberwithin{equation}{section}

\newcommand{\+}{\nobreak\hspace{.087em}\nopagebreak}

\def\Re{\mathop{\mathrm{Re}}}
\def\Im{\mathop{\mathrm{Im}}}

\def\D{\mathbb D}

\def\C{\mathbb C}
\def\Hol{{\sf Hol}}
\def\Aut{{\sf Aut}(\mathbb D)}

\def\N{\mathbb N}

\def\id{{\sf id}}

\newcommand{\hd}{\rho}
\newcommand{\phd}{\rho^*}

\newcommand{\Zn}{\mathcal{Z}_{\scriptscriptstyle\forall}}

\newcommand{\UH}{\mathbb{H}}
\newcommand{\UHR}{{\mathbb{H}_{\mathrm R}}}
\newcommand{\UD}{\mathbb{D}}
\newcommand{\UC}{\partial\UD}
\newcommand{\Complex}{\mathbb{C}}
\newcommand{\Real}{\mathbb{R}}
\newcommand{\Natural}{\mathbb{N}}
\newcommand{\di}{\mathop{\mathrm{d}}\nolimits}
\newcommand{\anglim}{\angle\lim}

\renewcommand{\ge}{\geqslant}
\renewcommand{\le}{\leqslant}
\renewcommand{\geq}{\geqslant}

\newcommand{\StepC}[2]{\medskip\noindent{\textsc{Case}~#1: #2.}}

\newenvironment{ourlist}{\begin{enumerate}[label={\bf (\arabic*)}, ref={\rm (\arabic*)}, left=0em]
\everydisplay{\makeatletter\def\@eqnum{\normalfont(\theequation)}\makeatother}}{\end{enumerate}}
\newenvironment{romlist}{\begin{enumerate}[label={\rm (\roman*)}, ref={\rm (\roman*)}, left=.7em]}{\end{enumerate}}

\newcommand{\dff}[1]{\textsl{#1}}

\newcommand\Wtilde[1]{\stackrel{\sim}{\smash{#1}\rule{0pt}{0.54ex}}}
\newcommand{\mtilde}[1]{\mathchoice{\,\widetilde{\!#1\!}\,}{\,\widetilde{\!#1\!}\,}%
{\raise.25ex\hbox{$\scriptstyle\Wtilde{#1}$}}%
{\Wtilde{#1}}}

\usepackage{xpatch}

\begin{document}
\title{On generalized embeddability of a non-elliptic self-map}

\author[M. D. Contreras]{Manuel D. Contreras $^\dag$}

\author[S. D\'{\i}az-Madrigal]{Santiago D\'{\i}az-Madrigal $^\dag$}
\address{Camino de los Descubrimientos, s/n\\
	Departamento de Matem\'{a}tica Aplicada II and IMUS\\
	Universidad de Sevilla\\
	Sevilla, 41092\\
	Spain.}\email{contreras@us.es} \email{madrigal@us.es}

\author[P. Gumenyuk]{Pavel Gumenyuk$^\ddag$}
\address{Pavel Gumenyuk: Department of
     Mathematics\\  Politecnico di Milano, via E. Bonardi 9\\ Milan 20133, Italy.}\email{pavel.gumenyuk@polimi.it}

\thanks{$^\dag$ Partially supported by Ministerio de Ciencia, Innovaci\'on y Universidades, Spain, projects PID2022-136320NB-I00 and PID2025-167700NB-I00. The authors thank IMUS-Mar\'{\i}a de Maeztu grant CEX2024-001517-M - Apoyo a Unidades de Excelencia Mar\'{\i}a de Maeztu for supporting this research, funded by MICIU/AEI/ 10.13039/501100011033.}
\thanks{$^\ddag$ Partially supported  by GNSAGA INdAM (\textit{Istituto Nazionale di Alta Matematica ``Francesco Severi''}) Italy}
\thanks{The authors declare that no artificial intelligence (AI) tools, generative AI, or AI-assisted technologies were used in the preparation, writing, or analysis of this manuscript.}

%\date\today


\begin{abstract}
We study the structure of the centralizer of a non-elliptic holomorphic self-map of the unit disc and its relation with continuous dynamical structures. We prove that if the identity is not isolated in the centralizer of a non-elliptic self-map $\varphi \in \Hol(\D)$, then $\varphi$ is univalent and its centralizer contains a non-trivial continuous one-parameter semigroup of holomorphic self-maps. This result provides a general criterion for the existence of continuous dynamics arising from discrete iteration.

Our approach is based on the theory of holomorphic models and on simultaneous linearization techniques for commuting self-maps. In particular, we exploit recent developments on simultaneous linearization for parabolic self-maps with positive hyperbolic step. As a consequence, we obtain new sufficient conditions for the embeddability of holomorphic self-maps into continuous semigroups.
\end{abstract}

\maketitle

\section{Introduction}\label{Introduction}

The iteration theory of holomorphic self-maps of the unit disc has long been a central topic in complex dynamics and geometric function theory. A particularly fruitful line of research concerns the algebraic and dynamical structure associated with a given holomorphic self-map, especially through its semigroup properties and its interaction with other commuting maps.

Given a holomorphic self-map $\varphi \in \Hol(\D)$, its centralizer is defined as
$$
\Zn(\varphi):=\{\psi \in \Hol (\D): \, \psi\circ\varphi=\varphi\circ \psi\}.
$$

The structure of $\Zn(\varphi)$ encodes important dynamical information about $\varphi$. In particular, the presence of non-trivial elements commuting with $\varphi$ often reflects hidden rigidity or additional symmetries in the underlying dynamics. Understanding when the centralizer is ``large''~--- for instance, when the identity is not isolated in $\Zn(\varphi)$~--- naturally leads to questions about the existence of continuous dynamical structures.

These considerations are closely related to the classical embedding problem: given $\varphi \in \Hol(\D)$, determine whether there exists a continuous one-parameter semigroup $(\phi_t)_{t\ge0}$ of holomorphic self-maps of the disc such that $\phi_{t_0}=\varphi$ for some $t_0>0$. This problem has been extensively studied and widely discussed in the literature, as it connects iteration theory, functional equations, infinitesimal generators, and the geometry of Koenigs domains.
Embeddability in terms of the self-map $\varphi$ is well understood in only a few settings.
For example, it is completely characterized for M\"obius transformations of~$\C$ mapping the unit disc into itself, see e.g. \cite[Theorems~9.3.4 and~9.3.5]{BCD-Book}.

In this paper we investigate non-elliptic self-maps of the unit disc from the perspective of their centralizers.
Our main result shows that the existence of holomorphic self-mappings $\psi\neq\id_\UD$ arbitrarily close to the identity map~$\id_\UD$ and, at the same time, commuting with a given non-elliptic self-map ${\varphi\in\Hol(\UD)}$ forces the presence of a continuous dynamical structure within~$\Zn(\varphi)$.
More precisely:
\begin{theorem}\label{Thm:embedding}
Let $\varphi \in \Hol(\D)$ be non-elliptic. If $\id_{\D}$ is not isolated in $\Zn(\varphi)$, then $\varphi$ is univalent and $\Zn(\varphi)$ contains a non-trivial continuous one-parameter semigroup~$(\phi_t)$.
\end{theorem}

Some partial results of this kind were previously known. With an \textit{a priori} assumption that $\varphi$ is univalent, \cite[Theorem~4.4]{CDG-Centralizer} states the same conclusion under a stronger hypothesis that $\id_\UD$ is not isolated in
$$
 \mathcal Z(\varphi):=\Zn(\varphi)\cap\big\{\psi\in\Hol(\UD):\text{$\psi$ is univalent in~$\UD$}\big\}.
$$
Without univalence assumptions in the hypothesis, Theorem~\ref{Thm:embedding} has been proved in~\cite{CDG-zero-h.step} for the special (and in many respects simpler) case of parabolic self-maps~$\varphi$ of zero hyperbolic step. The present work completes the picture in the positive hyperbolic step setting (both for the parabolic and hyperbolic cases), further clarifying the relationship between the size of the centralizer and the embeddability phenomenon.

The main tools to obtain this result have been developed recently in \cite{CDG-Centralizer, CDG-zero-h.step, CDG-positive-h.step}, where the existence of holomorphic functions that simultaneously linearize any two commuting parabolic maps is established (see Section \ref{SEC:Simultaneous linearization}).

The property of~$\varphi$ asserted in the conclusion of Theorem~\ref{Thm:embedding} can be regarded as a kind of generalized or weaker version of embeddability, which explains the title of the paper. It is worth mentioning that if the function $\varphi$ is \emph{hyperbolic}, then it belongs to the semigroup~$(\phi_t)$ provided by Theorem~\ref{Thm:embedding}. That is, there exists $t_{0}>0$ such that $\varphi=\phi_{t_{0}}$ (see  \cite[Proposition~6.9]{CDG-Centralizer}, but also it is a consequence of our proof of Theorem~\ref{Thm:embedding}).
At the same time, in the \emph{parabolic case,} the existence of such a semigroup~$(\phi_t)$ inside the centralizer does not automatically imply that $\varphi$ itself is an element of that semigroup. Indeed, examples show that a parabolic self-map may commute with a non-trivial semigroup without being one of its iterates;  see, for example, \cite[Example~8.4]{CDG-Centralizer}.

In other words, the ``generalized'' and ``standard'' notions of embeddability of a non-elliptic self-map are equivalent only in the hyperbolic case, and the hypothesis that $\id_\UD$ is not isolated in~$\Zn(\varphi)$, which is obviously a necessary condition of embeddability, fails to be sufficient in general. Nevertheless, we can provide an additional sufficient condition ensuring full embeddability. Namely, in combination with \cite[Theorem~1.4]{CDG-Embedding} and~\cite[Corollary~2.13]{CDG-Centralizer}, Theorem~\ref{Thm:embedding} immediately implies the following result.

\begin{corollary}\label{Cor:embedding}
Let $\varphi \in \Hol(\D)$ be a parabolic self-map having at least one boundary fixed point other than its Denjoy\,--\,Wolff point. Suppose that $\id_{\D}$ is not isolated in $\Zn(\varphi)$. Then there exists a non-trivial continuous one-parameter semigroup ${(\phi_t)\subset\Zn(\varphi)}$ such that  $\varphi=\phi_{t_0}$ for some~${t_0>0}$ (and hence, by the very definition,  $\varphi$ is embeddable). Moreover, up to a reparametrization, $(\phi_t)$ is the only non-trivial continuous one-parameter semigroup contained in~$\Zn(\varphi)$.
\end{corollary}
In Section~\ref{S_real}, we obtain another result of this kind, Corollary~\ref{CR_another-add-condition}. It states that if under the hypothesis of Theorem~\ref{Thm:embedding} we additionally suppose that ${\varphi\big((-1,1)\big)\subset(-1,1)}$, then $\varphi$ is embeddable.

The main objective of this paper is to prove Theorem~\ref{Thm:embedding}. The proof itself is given in Section~\ref{SEC:Proof}, which is preceded by two preparatory sections. In Section~\ref{S_prelim}, we introduce most of the terminology we use and provide the necessary basics of the theory of holomorphic models for self-maps of the unit disc. Next, in Section~\ref{SEC:Simultaneous linearization}, we collect the results on the simultaneous linearization of commuting holomorphic self-maps serving as the main tool in our proof of Theorem~\ref{Thm:embedding}.

In Section~\ref{S_real}, we state and prove the sufficient condition for embeddability mentioned above.
Finally, in the concluding Section~\ref{S_concluding}, we make a few observations related to the main result. In particular, we characterize infinitesimal generators of one-parameter semigroups contained in the centralizer of a given non-elliptic self-map and show that they form a topologically closed convex cone. We also notice that similarly to ``standard'' embeddability, the presence of a non-trivial continuous one-parameter semigroup in the centralizer~$\Zn(\varphi)$ affects the boundary behaviour of~$\varphi$, see Remark~\ref{RM_anglim}.

\section{Preliminaries and auxiliary results}\label{S_prelim}
Below we introduce some notation and basic theory used further in the paper. For  more details and for the proofs of the previously known results presented in this section without proof, we refer the interested readers to the recent monograph~\cite{Abate2}.

\addtocontents{toc}{\SkipTocEntry}
\subsection{Notation}\label{Notation}
As usual, we denote the unit disc by
${\UD:=\{z\in\C:|z|<1\}}$,
and we write $\UH:={\{w\in\C:\Im w>0\}}$ for the upper half-plane and $\UHR:={\{w\in\C:\Re w>0\}}$ for the right half-plane.

Furthermore, denote by $\Hol(D,E)$ the class of all holomorphic mappings of a domain $D\subset\C$ into a set $E\subset\C$. As usual, we endow $\Hol(D,E)$  with the topology of locally uniform convergence. In case $E=D$, we will write  $\Hol(D)$  instead of $\Hol(D,D)$.

For a self-map $\varphi:D\to D$ of a domain $D\subset\C$ and ${n\in\Natural}$, we denote by $\varphi^{\circ n}$ the $n$-th iterate of~$\varphi$, and let $\varphi^{\circ0}:=\id_D$, the identity map in $D$.

\addtocontents{toc}{\SkipTocEntry}
\subsection{Holomorphic self-maps of the unit disc}\label{SS_PRE-holo-selfmaps}
The study of the dynamics of a generic holomorphic self-map $\varphi$ of the unit disc $\mathbb{D}$, different from the identity map,  is a classical and well-established branch of Complex Analysis.
The central result in the area is the Denjoy\,--\,Wolff Theorem, which states that if $\varphi$ is different from an elliptic automorphism (i.e. not an automorphism of~$\UD$ possessing a fixed point in~$\UD$), then the sequence of the iterates $(\varphi^{\circ n})$ converges locally uniformly in~$\UD$ to a certain point~${\tau\in\overline{\mathbb{D}}}$.  This point  is called the \dff{Denjoy\,--\,Wolff point\/} of $\varphi$. If it belongs to~$\UD$, then it is the unique fixed point of~$\varphi$ in~$\UD$. Otherwise, i.e. when ${\tau\in\UC}$, the self-map $\varphi$ does not possess any fixed point in~$\UD$, but $\tau$ is still in a sense a ``special'' fixed point: it is the unique boundary fixed point at which the angular derivative $\varphi'(\tau)$ is finite and belongs to $(0,1]$. (For the definition and basic facts on boundary fixed points of holomorphic self-maps, see e.g. \cite[\S1.9]{BCD-Book}.)

According to the position of the Denjoy\,--\,Wolff point~$\tau$ and to the value of $\varphi'(\tau)$, known as the \dff{multiplier} of $\varphi$ at~$\tau$, the holomorphic self-maps $\varphi\in\Hol(\UD)$ different from elliptic automorphisms are divided into three categories. Namely, $\varphi$ is called:
\begin{itemize}
\item[(a)] \dff{elliptic\/} if $\tau\in\UD$,

\item[(b)] \dff{hyperbolic\/} if $\tau\in \partial \D$ and $\varphi'(\tau )<1$, and

\item[(c)] \dff{parabolic\/} if $\tau
\in \partial \D$ and $\varphi'(\tau )=1$.
\end{itemize}
 All elliptic automorphisms of~$\UD$, along with the identity mapping~$\id_\UD$,  are conventionally included in the category~(a) of elliptic self-maps.
Similarly, for an elliptic automorphism different from~$\id_\UD$, its Denjoy\,--\,Wolff point is defined to be its unique fixed point in~$\UD$.

Parabolic self-maps can have very different dynamical properties depending on the so-called \textit{hyperbolic step}.
Let $\varphi\in\Hol(\UD)$ be non-elliptic. Thanks to the Schwarz\,--\,Pick Lemma, for the orbit $\big(z_n\big):=\big(\varphi^{\circ n}(z_0)\big)$ of any point ${z_0\in\UD}$, there exists a finite limit $q(z_0):=\lim_{n\to+\infty} \hd_\D(z_{n},z_{n+1})$, where $\hd_\UD(\cdot,\cdot)$ stands for the hyperbolic distance in~$\UD$.
 It is known, see e.g. \cite[Corollary\,4.6.9]{Abate2}, that  either $q(z_0)>0$ for all~${z_0\in\UD}$, or $q\equiv0$ in~$\UD$.  The self-map~$\varphi$ is said to be of \dff{positive} or of \dff{zero hyperbolic step} depending on whether the former or the latter alternative occurs.
If $\varphi$ is
hyperbolic, then it is always of positive hyperbolic step. However, there exist parabolic self-maps of zero as well as of positive hyperbolic step.

\addtocontents{toc}{\SkipTocEntry}
\subsection{Commuting holomorphic self-maps}\label{SS_commuting}
It is clear that if two  holomorphic self-maps $\varphi,\psi\in\Hol(\UD)\setminus\{\id_\UD\}$ commute, i.e. ${\varphi\circ\psi}={\psi\circ\varphi}$, and if one of them is elliptic, then the other is also elliptic and they share the Denjoy\,--\,Wolff point. The situation is not so evident when we consider non-elliptic self-maps. In 1973, Behan \cite{Behan}, see also \cite[Section~4.10]{Abate2}, proved that if $\varphi$ and~$\psi$ are commuting non-elliptic self-maps of~$\UD$ and $\varphi$ is not a hyperbolic automorphism, then  $\varphi$ and~$\psi$ share the  Denjoy\,--\,Wolff point.

Later, Cowen~\cite[Corollary~4.1]{Cowen-comm} proved that if $\varphi$ and $\psi$ are two non-elliptic commuting holomorphic self-maps of~$\D$ and if $\varphi$  is hyperbolic, then $\psi$ is also hyperbolic (and thus, if $\varphi$ is parabolic, then $\psi$ is parabolic). In the special case $\varphi$ is a hyperbolic automorphism, by an old result of Heins \cite[Lemma~2.1]{Heins}, $\psi\in\Hol(\UD)\setminus\{\id_\UD\}$ commutes with~$\varphi$ if and only if $\psi$ is a hyperbolic automorphism itself with the same fixed points as~$\varphi$.

It is worth mentioning that the centralizer $\Zn(\varphi):={\{\psi \in \Hol (\D): \, \psi\circ\varphi=\varphi\circ \psi\}}$ of any holomorphic self-map ${\varphi:\UD\to\UD}$ is a unital topological semigroup w.r.t. composition and topology of locally uniform convergence in~$\UD$. Moreover, it is topologically closed:
\begin{lemma}\label{LM_closed}
 For any $\varphi\in\Hol(\UD)$, $\Zn(\varphi)$ is relatively closed in~$\Hol(\UD)$.
\end{lemma}
\begin{proof}
Let $(\psi_n)\subset\Zn(\varphi)$ be a sequence converging to some ${\psi\in\Hol(\UD)}$. Passing to the limit as~${n\to+\infty}$ in the equality ${\psi_n\circ\varphi}={\varphi\circ\psi_n}$, we see that, as desired, ${\psi\in\Zn(\varphi)}$.
\end{proof}

\addtocontents{toc}{\SkipTocEntry}
\subsection{Holomorphic models for holomorphic self-maps}\label{Sec:modelos}
An indispensable role in our study is played by the concept of a holomorphic model, which goes back to Pommerenke~\cite{Pom79}, Baker and Pommerenke~\cite{BakerPommerenke}, and Cowen~\cite{Cowen}. The terminology we use is mainly borrowed from~\cite{Canonicalmodel}.

\begin{definition}\label{DF_absorbing} Let  $\varphi\in\Hol(\D)$. A domain $V \subset \D$ is \dff{invariant} for $\varphi$ (or $\varphi$-invariant) if $\varphi(V)\subset V$; it is \dff{absorbing} for $\varphi$ (or $\varphi$-absorbing) if it is $\varphi$-invariant and
$$
\D=\bigcup_{n\in\Natural}\big(\varphi^{\circ n}\big)^{-1}(V).
$$
In other words, a $\varphi$-invariant domain is $\varphi$-absorbing if it eventually contains the orbit of any point of $\D$.
A simply connected domain $V\subset \D$ is called a \dff{fundamental domain} of $\varphi$ if it is $\varphi$-absorbing and $\varphi:V\to V$ is univalent.
\end{definition}

\begin{definition}\label{DF_holomorphic-model} A \dff{holomorphic semimodel} of $\varphi\in\Hol(\D)$  is a triple $\mathcal M:=(S,h,\alpha)$, where $S$ is a Riemann surface, $\alpha$ is an automorphism of $S$, and $h$ is a holomorphic map from $\UD$ into $S$ satisfying the following two conditions:
	\begin{enumerate}[left=2.5em]
		\item[(HM1)] $h\circ \varphi=\alpha \circ h$  {}~and
		\item[(HM2)] $S\,=\,\bigcup_{n\geq0}  \big(\alpha^{\circ n}\big)^{-1}(h(\D))$.
	\end{enumerate}
		A \dff{holomorphic model} of $\varphi\in \Hol (\D)$ is a semimodel $\mathcal M:=(S,h,\alpha)$ for which
	\begin{enumerate}[left=2.5em]
		\item[(HM3)] there exists a $\varphi$-absorbing domain $V\subset\UD$ in which $h$ is injective.
	\end{enumerate}
The Riemann surface $S$ is called the \dff{base space}, and the map $h$ is called the \dff{intertwining map} of the holomorphic model~$\mathcal M$.
\end{definition}

Given a holomorphic model $\mathcal M$ of~$\varphi$, every holomorphic semimodel of~$\varphi$ factorises via the model~$\mathcal M$ in the following sense.

\begin{lemma}[{\cite[Lemma 3.5.8 and Remark 3.5.6]{Abate2}}]\label{Le:morphism}
Assume that $\varphi\in\Hol(\UD)$ admits a holomorphic model $\mathcal M:=(S,h,\alpha)$.
If $(S_1,h_1,\alpha_1)$ is another holomorphic semimodel for~$\varphi$, then there exists  a unique surjective holomorphic map $\beta:S\to S_1$ such that ${h_1=\beta\circ h}$ and ${\beta\circ\alpha}={\alpha_1\circ\beta}$.
\end{lemma}

Every non-elliptic self-map $\varphi\in\Hol(\UD)$ admits an essentially unique holomorphic model. More precisely, the following fundamental theorem holds.

\begin{theorem}[\protect{\cite[Corollary~4.5.5]{Abate2}}] \label{Thm:uniqness} Every
non-elliptic self-map $\varphi\in\Hol(\UD)$ admits a holomorphic model. Moreover such a model is unique up to a model isomorphism; i.e., if $(S_1,h_1,\alpha_1)$ and $(S_2,h_2,\alpha_2)$ are holomorphic models for $\varphi$, then there exists a biholomorphic map $\beta$ of~$S_1$ onto~$S_2$ such that
	$$
	h_2=\beta\circ h_1,\quad \alpha_2=\beta\circ\alpha_1\circ \beta^{-1}.
	$$
\end{theorem}

The proof of the existence of holomorphic models depends strongly on the existence of absorbing sets (see, e.g., \cite[Theorem 3.5.10]{Abate2}).
In the next result we summarize what we need about the models for self-maps of the unit disc.
The type of a self-map is reflected in, and actually can be fully determined from the kind of holomorphic model $\varphi$ admits.  For an open interval $I\subset\Real$, we define
$$
 S_I:=\Real\times I=\{x+iy:x\in\Real,\,y\in I\}.
$$

\begin{theorem}[\protect{\cite{Cowen}, see also \cite[Theorem 4.6.8]{Abate2}}]\label{Thm:model}
Let $\varphi\in\Hol(\UD)$. The following statements hold.
\begin{ourlist}
	\item\label{IT_HM-hyp} $\varphi$ is a hyperbolic self-map with multiplier $\lambda\in(0,1)$ if and only if $\varphi$ admits a holomorphic model of the form $\mathcal M_\varphi:=(S_{I},h,z\mapsto z+1)$, where $I=(a,b)$ is a bounded open interval of length ${b-a=\pi/|\log\lambda|}$.
	\item\label{IT_HM-para-PHS} $\varphi$ is a parabolic self-map of positive hyperbolic step if and only if $\varphi$ admits a holomorphic model of the form ${\mathcal M_\varphi:=(S_{I},h,z\mapsto z+1)}$, where $I$ is an open unbounded interval different from the whole~$\Real$.
	\item\label{IT_HM-para-0HS} $\varphi$ is a parabolic self-map of zero hyperbolic step if and only if $\varphi$ admits a holomorphic model of the form ${\mathcal M_\varphi:=(\C,h,z\mapsto z+1)}$.	
\end{ourlist}	
\end{theorem}

\begin{remark}\label{RM_normalization}
In the above theorem, we may assume that:
\begin{itemize}
 \item[-]in cases~\ref{IT_HM-hyp} and~\ref{IT_HM-para-0HS}, $h(0)=0$;
 \item[-]in case~\ref{IT_HM-para-PHS}, $\Re h(0)=0$ and ${S_I=S_{(0,+\infty)}=\UH}$ or ${S_I=S_{(-\infty,0)}=-\UH}$.
\end{itemize}
Using the uniqueness part of Theorem~\ref{Thm:uniqness}, one can show (see e.g. \cite[Corollary 4.6.12]{Abate2} for details) that the above assumptions play the role of a normalization under which the holomorphic model $\mathcal M_\varphi$ for a given non-elliptic self-map~$\varphi$ is unique. Note that in case~\ref{IT_HM-para-PHS}, we can actually achieve ${S_I=\UH}$ by replacing, if necessary, $\varphi$ with $z\mapsto\overline{\varphi(\bar z)}$.
\end{remark}

\begin{definition}\label{DF_canonical}
 The unique holomorphic model $\mathcal M_\varphi$ of a non-elliptic self-map  $\varphi\in\Hol(\UD)$ defined in Theorem~\ref{Thm:model} and normalized as in Remark~\ref{RM_normalization} is called \dff{the canonical (holomorphic) model} for~$\varphi$. The intertwining map~$h$ of the canonical model~$\mathcal M_\varphi$ is called the \dff{Koenigs function}, and we denote it, from now on, by~$h_{\varphi}$. Finally, ${h_{\varphi}(\UD)}$ is called the \dff{Koenigs domain} of~$\varphi$  and we denote by  $S_\varphi$ the base space of~$\mathcal M_\varphi$.
\end{definition}

\begin{remark}\label{RM_Pommerenke}
It is known that the Koenigs function~$h_\varphi$ of any hyperbolic self-map~$\varphi\in\Hol(\UD)$ satisfies
\begin{equation*}
 \anglim_{z\to\tau} \Re h_\varphi(z)=+\infty,
\end{equation*}
where ${\tau\in\UC}$ stands for the Denjoy\,--\,Wolff point of~$\varphi$. This statement can be deduced, for example, from an old result by Pommerenke~\cite[Theorem~3\+(i)]{Pom79} with the help of the uniqueness theorem for solutions to Valiron's equation, see \cite[Proposition~6]{BracciPC}.
\end{remark}

\subsection{Continuous one-parameter semigroups of holomorphic self-maps of the unit disc}\label{SS_PRE-semigroups}
 We recall that
a family $(\phi_{t})_{t\geq 0}$  of holomorphic self-maps $\phi_{t}:\D\to \D$ is called a \dff{one-parameter semigroup} if it verifies the following two algebraic properties:
\begin{enumerate}
	\item[(i)] $\phi_{0}=\id_{\D}$;
	\item[(ii)] $\phi_{t+s}=\phi_{t} \circ \phi_{s}$ for every $t,s\geq 0$.
\end{enumerate}
A one-parameter semigroup $(\phi_t)$ is said to be \dff{continuous} if ${\phi_{t}\to \phi_{0}=\id_\UD}$ locally uniformly in~$\D$ as ${t\to 0^{+}}$. Finally, a one-parameter semigroup~$(\phi_t)$ is called \dff{non-trivial} if ${\phi_t\neq\id_\UD}$ for at least one value of~${t>0}$.

The state of the art of the theory of continuous one-parameter semigroups of holomorphic self-maps can be found in the recent monograph~\cite{BCD-Book}.
It is however worth mentioning here a couple of fundamental facts explicitly. It is known (\cite{BP}, see also \cite[Theorem~10.1.4]{BCD-Book}) that if a one-parameter semigroup~$(\phi_t)$ is continuous, then it is in fact differentiable in~$t$ and moreover, there exists a holomorphic function $G:\UD\to\C$, called the \dff{infinitesimal generator} of~$(\phi_t)$ such that $\di\phi_t(z)/\di t=G(\phi_t(z))$ for every~${z\in\UD}$ and~${t\ge0}$.

It follows from the continuous version of the Denjoy\,--\,Wolff Theorem (see e.g. \cite[Theorem~8.3.1]{BCD-Book}) that if the infinitesimal generator~$G$ does not vanish in~$\UD$, then all $\phi_t$'s except for~${\phi_0=\id_\UD}$ are non-elliptic. Moreover, in this case,  any antiderivative~$H$ of the function~$1/G$ is univalent in~$\UD$ and satisfies Abel's equation  ${H\circ\phi_t=H+t}$ for all~${t\ge0}$; see e.g. \cite[Definitions~9.2.2 and~9.3.9, and Theorem~10.1.4\+(3)]{BCD-Book}. The function~$H$ is called a \dff{Koenigs function} of the one-parameter semigroup~$(\phi_t)$.

\section{Simultaneous linearization}\label{SEC:Simultaneous linearization}
An important role in the study of commuting non-elliptic self-mappings is played by the \textit{simultaneous linearization problem}
\begin{equation}\label{EQ_SL}
h\circ\varphi=h+1,\qquad h\circ\psi=h+c,
\end{equation}
which is known to have a solution for any non-elliptic self-map~${\varphi\in\Hol(\UD)}$ and any $\psi\in\Zn(\varphi)$. More precisely, the following basic result holds.
{
\makeatletter
\xpatchcmd{\@thm}{\thm@headpunct{.}}{\thm@headpunct{}}{}{}
\makeatother
\begin{theorem}\label{TH_SL0}
{\rm(see {\protect{\cite[Theorem~4.1]{CDG-zero-h.step}} and \cite[Theorem~4.2]{CDG-positive-h.step} for the parabolic case} and \protect{\cite[Theorem~3.1]{Cowen}} for the hyperbolic case){\bf.}}
Let $\varphi\in\Hol(\UD)$ be a non-elliptic self-map and let $\psi\in\Hol(\UD)$ commute with~$\varphi$. Then there exists a unique constant $c\in\C$ such that the simultaneous linearization problem~\eqref{EQ_SL} admits a holomorphic solution ${h:\UD\to\C}$.
\end{theorem}
For $\psi=\id_\UD$, the statement of this theorem becomes clear if we recall that the Koenigs function~$h_\varphi$ satisfies the first equation in~\eqref{EQ_SL}.
The uniqueness of~$c$ for parabolic self-mappings follows directly from~\cite[Corollary~4.9]{CDG-positive-h.step}, while for the hyperbolic case, it can be established by using literally the same argument as in \cite[Remark~6.2]{CDG-Centralizer}.

\begin{definition}\label{DF_SimLinCoeff}
The unique value of~$c\in\C$ for which~\eqref{EQ_SL} admits a holomorphic solution ${h:\UD\to\C}$ is called the \dff{simultaneous linearization coefficient} of~$\psi\in\Zn(\varphi)$ w.r.t. the non-elliptic self-map~$\varphi$. In what follows we will denote it by~$c_{\varphi,\psi}$.
\end{definition}

In part of the cases, the existence of~$h$ in Theorem~\ref{TH_SL0} is established by considering the Koenigs function~$h_\varphi$ of the self-map~$\varphi$. Namely, we have the following result.

{
\makeatletter
\xpatchcmd{\@thm}{\thm@headpunct{.}}{\thm@headpunct{}}{}{}
\makeatother
\begin{proposition}{\rm(see \protect{\cite[Theorem~4.1]{CDG-zero-h.step}} for the parabolic case and \protect{\cite[Theorem~3.1]{Cowen}} for the hyperbolic case){\bf.}} \label{PR_hyperbolic_and_zero_step}
 Let $\varphi\in\Hol(\UD)$ be a hyperbolic or parabolic self-map of zero hyperbolic step with canonical model $(S_\varphi,h_\varphi,z\mapsto z+1)$, and let $\psi\in\Zn(\varphi)$. Then there exists a number ${c\in \C}$ such that ${h_{\varphi}\circ\psi =  h_{\varphi}+ c}$.
\end{proposition}}

\begin{remark}\label{RM_from-Cowen-to-us}
The reference to \cite[Theorem~3.1]{Cowen} needs some explanation.
To be precise, what Cowen proved in~\cite{Cowen} for hyperbolic self-maps~$\varphi$ reads in our terminology as follows: \textit{if ${\psi\in\Zn(\varphi)}$, then ${h_\varphi\circ\psi}={\alpha\circ h_\varphi}$ for a suitable ${\alpha\in\mathsf{Aut}(S_\varphi)}$ commuting with ${z\mapsto z+1}$.} Since $S_\varphi$ is a horizontal strip, one can conclude (using e.g. Heins' Lemma, see Section~\ref{SS_commuting}) that $\alpha(z)={z+c}$ for all~${z\in S_\varphi}$ and some constant~${c\in\Real}$.
\end{remark}
 Clearly, the constant in the above remark and proposition is $c=c_{\varphi,\psi}$.  We can say a bit more about this constant in the hyperbolic case.
\begin{lemma}\label{LM_c-positive}
Let $\varphi\in\Hol(\UD)$ be hyperbolic but not an automorphism, and let $\psi\in{\Zn(\varphi)\setminus\{\id_\UD\}}$. Then $c_{\varphi,\psi}$  is strictly positive.
\end{lemma}
\begin{proof}
By the results of Behan and Cowen (see Section~\ref{SS_commuting}), $\psi$ is a hyperbolic self-map of~$\UD$ having the same Denjoy\,--\,Wolff point~$\tau$ as the self-map~$\varphi$. Fix some ${z_0\in\UD}$ and let ${z_n:=\psi^{\circ n}(z_0)}$. Then ${h_\varphi(z_n)=h_\varphi(z_0)+nc_{\varphi,\psi}}$ for all ${n\in\Natural}$. By Remark~\ref{RM_Pommerenke},
$
  \anglim_{z\to\tau} \Re h_\varphi(z)=+\infty.
$
Since $(z_n)$ converges to~$\tau$ non-tangentially, see e.g. \cite[Pro\-position~4.3.2]{Abate2}, it follows that $${\Re\big(h_\varphi(z_0)+nc_{\varphi,\psi}\big)}={\Re h_\varphi(z_n)\to+\infty}\quad\text{ as $~{n\to+\infty}$.}$$  Therefore ${\Re c_{\varphi,\psi}>0}$, and it only remains to recall that ${c_{\varphi,\psi}\in\Real}$ by Remark~\ref{RM_from-Cowen-to-us}.
\end{proof}

In the remaining case, i.e. when $\varphi$ is a parabolic self-map of positive hyperbolic step, the situation is more complicated. If ${c_{\varphi,\psi}\in\Real}$, then again the Koenigs map~$h_\varphi$ solves the simultaneous linearization problem~\eqref{EQ_SL}, see \cite[Remark~4.20]{CDG-positive-h.step}. However, this is not necessarily the case if ${c_{\varphi,\psi}\not\in\Real}$. Nevertheless, any solution~$h$ to~\eqref{EQ_SL} factorizes via~$h_\varphi$. Moreover, summarizing some of the results proved in~\cite[Section~4]{CDG-positive-h.step}, see in particular \cite[Propositions~4.3, 4.5, and~4.7]{CDG-positive-h.step}, we have the following.
\begin{theorem}\label{TH_SL-positive-step}
Let $\varphi\in\Hol(\UD)$ be parabolic of positive hyperbolic step and let~$\psi\in\Zn(\varphi)$. Then the simultaneous linearization problem~\eqref{EQ_SL} admits a solution of the form $h={\beta\circ h_\varphi}$, where ${\beta:S_\varphi\to\C}$ is a holomorphic function satisfying the following identities:
\begin{subequations}\label{EQ_beta(sys)}
	\begin{align}
		\label{EQ_beta(a)} &\beta(w+1)=\beta(w)+1,\quad \text{for all $w\in S_\varphi$};\\
		\label{EQ_beta(b)} &\beta \circ h_{\varphi}\circ\varphi = \beta \circ h_{\varphi} + 1;\\
		\label{EQ_beta(c)} &\beta \circ h_{\varphi}\circ\psi = \beta \circ h _{\varphi}+ c_{\varphi,\psi};\\
        \label{EQ_beta(d)} &\beta(w)-w \to 0\quad\text{as~}~|\Im w|\to+\infty,~w\in S_\varphi.
	\end{align}
   \end{subequations}
Moreover, if $c_{\varphi,\psi}\in\C\setminus{\mathbb Q}$, then the function~$\beta\in\Hol(S_\varphi,\C)$ satisfying~\eqref{EQ_beta(sys)} is unique and every solution to~\eqref{EQ_SL} coincides with ${\beta\circ h_\varphi}$ up to an additive constant.
\end{theorem}

\begin{definition}\label{DF_SimLinFunc}
If $\varphi$ and~$\psi$ are as in Theorem~\ref{TH_SL-positive-step} and ${c_{\varphi,\psi}\in\C\setminus{\mathbb Q}}$, then we denote the unique solution to~\eqref{EQ_beta(sys)} by~$\beta_{\varphi,\psi}$. If ${c_{\varphi,\psi}\in{\mathbb Q}}$, then among the solutions we choose~--- and denote it again by~$\beta_{\varphi,\psi}$~--- the identity map~$\id_{S_\varphi}$.
In both cases, we call $\beta_{\varphi,\psi}$ the \dff{simultaneous linearization function} associated with~$\varphi$ and~$\psi$.
\end{definition}

Another tool, which plays an important role in the proof of our main result, is the notion of the semimodel morphism associated to a commuting pair of non-elliptic self-maps.

\begin{lemma}[\protect{\cite[Lemma 4.16]{CDG-positive-h.step}}]\label{LM_NSemimodel}
Let $\varphi\in\Hol(\D)$ be a non-elliptic self-map with canonical holomorphic model $\mathcal M_{\varphi}:={\big(S_\varphi,h_\varphi,z\mapsto z+1\big)}$ and let ${\psi\in\Zn(\varphi)}$. Then
\begin{equation*}%\label{EQ_NSemimodel1}
       S^{\psi}:=\bigcup_{n\geq0}\left(h_\varphi(\psi(\D))-n\right)
\end{equation*}
is a subdomain of the domain~$S_\varphi$. Moreover, the triple
\begin{equation}\label{EQ_NSemimodel2}
\mathcal M_\varphi^\psi:=\big(S^\psi,h_\varphi\circ\psi,z\mapsto z+1\big)
\end{equation}
 is a semimodel of $\varphi$. 	
\end{lemma}

\begin{definition}\label{DF_semimodel-morphism}
It is known (\cite[Definition~3.5.5 and Lemma~3.5.8]{Abate2}) that there exists a unique morphism~$g_\psi$ of the canonical model~$\mathcal M_\varphi$ to the semimodel $\mathcal M_\varphi^\psi$ defined by~\eqref{EQ_NSemimodel2}. We call~$g_\psi$ the \dff{semimodel morphism associated with~$\psi$}.
\end{definition}

\begin{remark}\label{LM_beta-Koenigs}
  Given $\psi\in\Hol(\UD)\setminus\{\id_\UD\}$ commuting with a parabolic self-map $\varphi\in\Hol(\UD)$ of positive hyperbolic step, the simultaneous linearization function $\beta_{\varphi,\psi}$ is closely related to the semimodel morphism~$g_\psi$ associated with~$\psi$. Namely, if we suppose for simplicity that $S_\varphi=\UH$, then ${\beta_{\varphi,\psi}=c_{\varphi,\psi}h_{g_\psi}+b_\psi}$, where $h_{g_\psi}$ is the Koenigs function of~$g_\psi$ and ${b_\psi\in\C}$ is a suitable constant. In case ${c_{\varphi,\psi}\not\in\Real}$, this relation was established in our proof of the uniqueness for solutions to~\eqref{EQ_beta(sys)}, see \cite[Proof of Propositions~4.5 and~4.7]{CDG-positive-h.step}. If ${c_{\varphi,\psi}\in\Real}$, it suffices to notice that in such a case, $\beta_{\varphi,\psi}=\id_{\UH}$ by the very definition and that ${h_{g_\psi}=c_{\varphi,\psi}^{-1}\id_\UH}$ because $g_\psi(w)={w+c_{\varphi,\psi}}$ for all ${w\in\UH}$ by \cite[Remark~4.20]{CDG-positive-h.step}.
\end{remark}
\begin{remark}\label{RM_beta-g_psi}
Moreover, the semimodel morphism~$g_\psi$ is a holomorphic self-map of~$S_\varphi$ and it satisfies the following relations
\begin{align*}
  & g_\psi\circ h_\varphi = h_\varphi\circ\psi,
      \qquad g_\psi\circ(w\mapsto w+1)=(w\mapsto w+1)\circ g_\psi,\\
  &\text{and}~\,~ \beta_{\varphi,\psi}\circ g_\psi=\beta_{\varphi,\psi}+c_{\varphi,\psi},
\end{align*}
where the last equality follows directly from the previous remark; for the other two equalities, see \cite[Remark~4.18]{CDG-positive-h.step}.
\end{remark}

\section{Proof of Theorem \ref{Thm:embedding}}\label{SEC:Proof}

\subsection{The hyperbolic case}\label{SS_hyperbolic}
Let $\varphi\in\Hol(\UD)$ be a hyperbolic self-map. Since every hyperbolic automorphism is embeddable, see e.g. \cite[Proposition~6.9\+(A)]{CDG-Centralizer}, we may assume that $\varphi$ is not an automorphism.

Suppose further that ${\Zn(\varphi)\setminus\{\id_\UD\}}$ contains a sequence $(\psi_n)$ converging to~$\id_\UD$. We are going to show that in such a case, $\Zn(\varphi)$ contains a non-trivial continuous one-parameter semigroup~$(\phi_t)$. As a byproduct of the proof, we will see that in fact, $\varphi=\phi_1$ and hence, $\varphi$ is embeddable.

According to Proposition~\ref{PR_hyperbolic_and_zero_step} and Lemma~\ref{LM_c-positive}, there exists a sequence ${(c_n)\subset(0,+\infty)}$ such that ${h_\varphi\circ\psi_n=h_\varphi+c_n}$ for all ${n\in\Natural}$. Consider the functions
$$
  G_n(z):=\frac{\psi_n(z)-z}{c_n},\quad z\in\UD,~\,n\in\Natural.
$$
By \cite[Corollaries~10.2.7 and~10.2.11]{BCD-Book}, these functions are infinitesimal generators.

Observe that $\psi_n(z)-z$ does not vanish in~$\UD$ because the self-maps $\psi_n$'s are hyperbolic (see Section~\ref{SS_commuting}). Since for each $z\in\UD$, ${\psi_n(z)\to z}$ as ${n\to+\infty}$, we have
$$
  \frac{h_\varphi(\psi_n(z))-h_\varphi(z)}{\psi_n(z)-z}\to h_\varphi'(z) \quad\text{as ${n\to+\infty}$}.
$$
Combining this with the identity
$$
  G_n(z)\,\frac{\,h_\varphi(\psi_n(z))-h_\varphi(z)\,}{\psi_n(z)-z}~=~
   \frac{\psi_n(z)-z}{c_n}\,\frac{c_n}{\psi_n(z)-z}=1\quad\text{for all~$z\in\UD$ and all~$n\in\Natural$},
$$
we see that $G_n(z)\to1/h_\varphi'(z)$ as ${n\to+\infty}$ for each ${z\in\UD}$ such that ${h_\varphi'(z)\neq0}$. According to \cite[Proposition~10.5.2]{BCD-Book}, it follows that $(G_n)$ has a subsequence converging locally uniformly in~$\UD$ to an infinitesimal generator. Therefore, ${G(z):=1/h_\varphi'(z)}$ is an infinitesimal generator in~$\UD$.

Denote by $(\phi_t)$ the one-parameter semigroup generated by~$G$.
Since $G$ does not vanish in~$\UD$, $(\phi_t)$ is a non-trivial non-elliptic semigroup and its Koenigs function $H$ is an antiderivative of $1/G$, see Section~\ref{SS_PRE-semigroups}. Thus, $H$ coincides, up to a constant, with~$h_\varphi$. Since the Koenigs function of a one-parameter semigroup is univalent, see e.g. \cite[Definitions~9.2.2 and~9.3.9]{BCD-Book}, it follows that ${\varphi=h_\varphi^{-1}\circ(h_\varphi+1)=\phi_1}$.

The desired conclusion follows now from the facts that elements of every continuous one-parameter semigroup are univalent (see e.g. \cite[Theorem~8.1.17]{BCD-Book}) and that they commute with each other (see condition~(ii) in Section~\ref{SS_PRE-semigroups}). \qed

\begin{remark}
In case of a hyperbolic self-map $\varphi\in\Hol(\UD)\setminus\Aut$ satisfying the hypothesis of Theorem~\ref{Thm:embedding}, the centralizer $\Zn(\varphi)$ can be completely described. Namely,   ${\Zn(\varphi)=\{\phi_t:t\ge0\}}$, that is, $\varphi$ commutes with no holomorphic self-maps of~$\UD$ other than the elements of the one-parameter semigroup~$(\phi_t)$ constructed in the above proof.
Indeed, if $\psi\in\Zn(\varphi)\setminus\{\id_\UD\}$, then by Proposition~\ref{PR_hyperbolic_and_zero_step} and Lemma~\ref{LM_c-positive}, we have
$$
  \psi=h_\varphi^{-1}\circ(h_\varphi+t)=\phi_t \quad \text{with~$~{t:=c_{\varphi,\psi}>0}$,}
$$
where we have used the fact that the Koenigs function of a univalent self-map is also univalent, see e.g. \cite[Lemma~3.5.4\+(iv)]{Abate2}.

The situation is quite different in the parabolic case, see e.g. examples in~\cite[Section~8]{CDG-Centralizer} and Example~\ref{EX_example}.
\end{remark}

\subsection{The parabolic case}
We start by introducing a construction, which will allow us to reduce the proof of Theorem~\ref{Thm:embedding} for parabolic self-maps to the case when~$\varphi$ is univalent.

Let $\varphi\in\Hol(\UD)$ be a parabolic self-map. Define
$$
 V_{\varphi}:=\big\{ z\in \mathbb{D}:\phd_\UD(z,\varphi
(z))<\tfrac13 \big\},
$$
where $\phd_\UD$ stands for the pseudo-hyperbolic distance in~$\UD$. It is known (see \cite[Theorem~5.1]{CDP}) that $V_\varphi$ is a fundamental domain for~$\varphi$, that is, $V_\varphi$ is simply connected and $\varphi$-absorbing, and ${\varphi:V_\varphi\to V_\varphi}$ is univalent. Fix a conformal map~$f$ of~$\D$ onto~$V_\varphi$. Then
\begin{equation}\label{EQ_varphi-tilde}
 \widetilde \varphi:=f^{-1}\circ\varphi\circ f
\end{equation}
is a well-defined univalent self-map~$\widetilde \varphi:\UD\to\UD$.

Moreover, by \cite[Lemma~2.10]{CDG-zero-h.step}, $\psi(V_\varphi)\subset V_\varphi$ for any ${\psi\in\Zn(\varphi)}$. As a result, there is a natural map~$\mathcal{R}_{\varphi}$ from $\Zn(\varphi)$ to $\Zn(\widetilde\varphi)$ defined by
\begin{equation}\label{EQ_the-tilde-map}
 \mathcal{R}_{\varphi}:\Zn(\varphi)\ni\psi\mapsto\widetilde \psi:=f^{-1}\circ\psi\circ f\in\Zn(\widetilde\varphi).
\end{equation}

\begin{lemma}\label{Thm:closedtilde}
 The map $\mathcal{R}_{\varphi}:\Zn(\varphi)\to\Zn(\widetilde\varphi)$ defined above is an injective homomorphism of unital semigroups, i.e. $\mathcal{R}_{\varphi}(\id_\UD)=\id_\UD$ and $\mathcal{R}_{\varphi}(\psi_1\circ\psi_2)={\mathcal{R}_{\varphi}(\psi_1)\circ \mathcal{R}_{\varphi}(\psi_2)}$ for any ${\psi_1,\psi_2\in\Zn(\varphi)}$. Moreover, $\mathcal{R}_{\varphi}$ is continuous and closed w.r.t. the topologies in $\Zn(\varphi)$ and $\Zn(\widetilde\varphi)$ inherited from~$\Hol(\UD)$.
\end{lemma}

\begin{proof}
It is straightforward to check that $\mathcal{R}_{\varphi}$ is a homomorphism and that it is continuous. The injectivity is an easy consequence of the Identity Principle.

Let us show that the map $\mathcal{R}_{\varphi}$ is closed.
To this end, consider a sequence ${(\psi_n)\subset\Zn(\varphi)}$ such that its image w.r.t. $\mathcal{R}_{\varphi}$, i.e. the sequence $(\widetilde \psi_n)$, converges locally uniformly in $\D$ to some ${\phi\in \Zn(\widetilde \varphi)}$.

We have to show that $\phi$ can be obtained as the image under~$\mathcal{R}_{\varphi}$ of a self-map belonging to the closure of ${F:=\{\psi_n:n\in\N\}}$ in $\Zn(\varphi)$.
Since $F$ is a relatively compact subset of $\Hol(\D,\C)$, we can take a subsequence $(\psi_{n_k})$ converging locally uniformly in~$\D$ to some ${\psi\in\Hol(\D,\C)}$.
Using the Open Mapping Theorem, we see that either ${\psi\in\Hol(\D)}$ or ${\psi\equiv \sigma}$ for some ${\sigma\in\partial\D}$. The latter alternative is not possible because, for every $z\in V_\varphi$,
	$$
	\sigma=\lim_{k\to\infty}\psi_{n_k}(z)=f\circ \phi\circ f^{-1}(z)\in\D.
	$$
As a consequence, by Lemma~\ref{LM_closed}, $\psi\in\Zn(\varphi)$. Appealing to the continuity of $\mathcal{R}_{\varphi}$, we conclude that $(\widetilde \psi_{n_k})$ converges locally uniformly in $\D$ to $\widetilde \psi=\mathcal R_\varphi(\psi)$. Thus, $\phi=\mathcal R_\varphi(\psi)$ as desired.
\end{proof}

Another important toolset for the proof of our main result is provided by the theory of normal families. Therefore, it seems worth quoting here some basics of this theory. For a detailed account on the topic, we refer interested readers to the monograph~\cite{NormalFamilies}.

Recall that a family~$\mathcal F\subset\Hol(D,\C)$ is said to be \dff{normal} in a domain~$D$ if every sequence ${(f_n)\subset \mathcal F}$ has a subsequence converging locally uniformly in~$D$ either to some ${f\in\Hol(D,\C)}$ or to~$\infty$. Montel's Criterion (see e.g. \cite[\S{}II.7, Theorem~1]{Goluzin} or \cite[p.\,54]{NormalFamilies}) implies that if ${D\subset\C}$ is a domain with $\C\setminus D$ containing at least two different points, then $\Hol(D)$ is a normal family.
\begin{remark}\label{RM_univ-normality}
If all elements of a family
$\mathcal F\subset\Hol(D,\C)$ are univalent in~$D$ and if there are two distinct points $z_0,z_1\in\UD$ such that the sets ${\{f(z_0):f\in\mathcal F\}}$ and ${\{f(z_1):f\in\mathcal F\}}$ are bounded, then $\mathcal F$ is normal in~$\UD$. This follows e.g. from  \cite[Theorems~A.5 and~A.6]{NormalFamilies}. Note also that adding any finite number of functions to a given family does not alter its normality. Moreover, it is known that $\mathcal F\subset\Hol(D,\C)$ is normal in~$D$ if and only if for any ${z_0\in D}$ there exists a subdomain of~$D$ containing~$z_0$  in which $\mathcal F$ is normal, see e.g. \cite[Theorem~2.1.2]{NormalFamilies}.
\end{remark}

Several times during the proof of Theorem~\ref{Thm:embedding} we will use the following well-known corollary of Hurwitz's Theorem, see e.g. \cite[p.\,4]{Duren}. (Alternatively, it can be deduced using \cite[\S\,I.1, Theorem~2]{Goluzin}.)
\begin{theorem}\label{TH_Hurwitz}
Let $f_n$ be a sequence of holomorphic functions converging locally uniformly in a domain~$D$ to a holomorphic function~$f$. The following statements hold:
\begin{romlist}
\item if all the functions $f_n$ are univalent in~$D$, then $f$ is either constant, or univalent~in~$D$;
\item if $f_n(D)\subset D$ for all~${n\in\Natural}$ and $f$ is not constant, then $f(D)\subset D$.
\end{romlist}
\end{theorem}

We will also need one auxiliary lemma. Let $a\in\Real$. Denote by $\rho_a^*(\cdot,\cdot)$ the pseudo-hyperbolic distance in the half-plane $\Pi_a:={\{w:\Im w>a\}}$.

\begin{lemma}\label{LM_Growth}
Let $\beta:\Pi_a\to\C$ be a holomorphic univalent function such that $\beta(w+1)={\beta(w)+1}$ for all ${w\in\Pi_a}$. Then for any given pair of points ${w_0,w_1\in\Pi_a}$, we have
\begin{equation}\label{EQ_Growth}
   \rho_a^*(w_0,w_1)\le\frac{4\rho}{(1-\rho)^2}\big|\beta(w_1)-\beta(w_0)\big|,\quad \text{where~}~\rho:=\rho_a^*(w_0,w_0+1).
\end{equation}
\end{lemma}
\begin{proof}
Let $g(z):=\Re w_0+i(\Im w_0-a)\frac{1+z}{1-z}+ia$, $z\in\UD$, be a conformal map of~$\UD$ onto~$\Pi_a$ with ${g(0)=w_0}$. Denote $\mu:=\big|\beta'(w_0)g'(0)\big|$. Applying the Growth Theorem (see e.g. \cite[Theorem~2.6]{Duren}) to the function
$
   f:=\big(\beta\circ g-\beta(w_0)\big)/(\beta'(w_0)g'(0)),
$
we obtain
$$
 1=|\beta(w_0+1)-\beta(w_0)|\le\mu\rho/(1-\rho)^2
$$
and
$$
  |\beta(w_1)-\beta(w_0)|\ge\frac{\mu\rho_a^*(w_0,w_1)}{\big(1+\rho_a^*(w_0,w_1)\big)^2}\ge
      \frac{\mu\rho_a^*(w_0,w_1)}{4}.
$$
These two inequalities easily imply~\eqref{EQ_Growth}.
\end{proof}

\begin{proof}[Proof of Theorem~\ref{Thm:embedding} for parabolic self-maps] As we mentioned in the introduction, for parabolic self-maps of zero hyperbolic step, the result is already known, see \cite[Theorem~4.4]{CDG-zero-h.step}. So in what follows we suppose that $\varphi\in\Hol(\UD)$ is a parabolic self-map of positive hyperbolic step.

We divide the proof into two parts. In Step~1, we will prove Theorem~\ref{Thm:embedding} under the additional assumption that~$\varphi$ is univalent. In Step~2, we will use Lemma~\ref{Thm:closedtilde} to extend it to the general case.

\medskip\noindent{\textsf{STEP 1.}}  Assume that $\varphi$ is univalent in~$\UD$ and let $(\psi_n)\subset{\Zn(\varphi)\setminus\{\id_\UD\}}$ be a sequence converging locally uniformly in~$\UD$ to the identity map. Note that the mappings~$\psi_n$ do not have to be univalent.

We can suppose without loss of generality that the canonical model of~$\varphi$ is of the form ${(\UH,h_\varphi,z\mapsto z+1)}$. Otherwise, i.e. if ${S_\varphi=-\UH}$, then we replace~$\varphi$ with ${z\mapsto\overline{\varphi(\overline{z})}}$.

Consider the semimodel morphisms $g_n:=g_{\psi_n}$ associated with $\psi_n$'s, ${n\in\Natural}$, see Definition~\ref{DF_semimodel-morphism}. According to Remark~\ref{RM_beta-g_psi}, ${g_n(\UH)\subset\UH}$ and ${g_n\circ h_\varphi}={h_\varphi\circ\psi_n}$.
The univalence of~$\varphi$ implies that its Koenigs function~$h_\varphi$ is also univalent in~$\UD$, see e.g. \cite[Lemma~3.5.4\+(iv)]{Abate2}. Hence, we can write $g_n|_\Omega={h_\varphi\circ\psi_n\circ h_\varphi^{-1}}$, where $\Omega:=h_\varphi(\UD)$. Since the functions $g_n$ map~$\UH$ into itself, they form a normal family in~$\UH$. Recalling that $(\psi_n)$ converges to~$\id_\UD$ and using the Vitali\,--\,Porter Theorem, see e.g. \cite[\S3.2]{NormalFamilies}, we therefore may conclude  that $(g_n)$  converges locally uniformly in~$\UH$ to~$\id_\UH$. As a consequence, $g_n$'s are eventually (i.e. for all~$n$ large enough) univalent in any domain compactly contained in~$\UH$. We are going to prove a stronger statement: for any~${\varepsilon>0}$, the functions $g_n$ are eventually univalent in $\Pi_\varepsilon:={\{w:\Im w>\varepsilon\}}$ for any~${\varepsilon>0}$.

For each ${n\in\Natural}$, denote  $c_n:=c_{\varphi,\psi_n}$. By \cite[Remark~4.23]{CDG-positive-h.step}, $c_n\neq0$. Moreover, as we have shown in \cite[proof of Proposition~4.3]{CDG-positive-h.step},  either ${c_n\in\Real}$ and ${g_n(w)={w+c_n}}$ for all~${w\in\UH}$, or ${c_n\in\UH}$ and then there exists ${F_n\in\Hol(\UD,\UH)}$ with ${F_n(0)=c_n}$ such that $g_n(w)={w+F_n(e^{2\pi i w})}$ for all~${w\in\UH}$. In both cases, we can write
\begin{equation}\label{EQ-from-g-to-f}
  e^{2\pi i g_n(w)}=f_n\big(e^{2\pi i w}\big)
\end{equation}
for all~${w\in\UH}$ and a suitable $f_n\in\Hol(\UD)$ with $f_n(0)=0$. Indeed, in the latter case, this follows from~\cite[Proposition~7.3\+(4)]{CDG-positive-h.step}, while in the former case, our claim holds trivially if we set $f_n(\zeta):=e^{2\pi i c_n}\zeta$ for all ${\zeta\in\UD}$.

Taking into account that $(g_n)$ converges to the identity map and that $f_n$'s form a normal family in~$\UD$, equality~\eqref{EQ-from-g-to-f} allows us to conclude that $(f_n)$ converges locally uniformly in~$\UD$ to~$\id_\UD$.

Since by~\eqref{EQ-from-g-to-f} we have
$$
   g'_n(w)=e^{2\pi i w}f'_n\big(e^{2\pi i w}\big)\big/{f_n\big(e^{2\pi i w}\big)}\quad \text{for all~$~w\in\UH~$ and all~$~n\in\Natural$,}
$$
it follows that for any $\varepsilon>0$, the sequence $(g_n')$ converges uniformly in~$\Pi_\varepsilon$ to~$1$. In particular, given an ${\varepsilon>0}$, we have ${\Re g_n'(w)>0}$ for all ${w\in\Pi_\varepsilon}$ and all~$n$ large enough. By the Noshiro\,--\,Warschawski Criterion (see e.g. \cite[Theorem~2.16]{Duren}) it follows that $g_n$ is univalent in~$\Pi_\varepsilon$ for all such $n$'s.

It further follows that the functions $\beta_n:=\beta_{\varphi,\psi_n}$ are eventually univalent in~$\Pi_\varepsilon$ for any ${\varepsilon>0}$. Indeed, by Remark~\ref{LM_beta-Koenigs}, for each~${n\in\Natural}$, $\beta_n=c_{\varphi,\psi_n}h_{g_n}+b_n$, where $h_{g_n}$ is the Koenigs function of~$g_n$ and ${b_n\in\C}$ is a constant.
By \cite[Proposition~7.3\+(2)]{CDG-positive-h.step}, we have ${g_n(\Pi_\varepsilon)\subset \Pi_\varepsilon}$ for all ${n\in\Natural}$. Since $g_n$ is univalent in~$\Pi_\varepsilon$ for all ${n\in\Natural}$ large enough, by \cite[Lemma~3.5.4\+(iv)]{Abate2} the same holds for~$h_{g_n}$ and hence also for the functions~$\beta_n$. In combination with the equality ${\beta_n(i+1)-\beta_n(i)=1}$, this fact implies that the functions $\beta^*_n:={\beta_n-\beta_n(i)}$ form a normal family in~$\Pi_\varepsilon$ for any ${\varepsilon\in(0,1)}$ and hence also in the whole~$\UH$; see Remark~\ref{RM_univ-normality}.  Therefore, passing to a subsequence, we may suppose that $(\beta^*_n)$ converges locally uniformly in~$\UH$ to some function~$\beta_*$. Since ${\beta_n^*(i)=0}$ and ${\beta_n^*(1+i)=1}$ for all ${n\in\Natural}$, the limit function~$\beta_*$ is not constant. Bearing this in mind and using Theorem~\ref{TH_Hurwitz}, we see that $\beta_*$ is univalent in~$\UH$.

Recall that $\beta_n\circ g_n=\beta_n+c_n$, see Remark~\ref{RM_beta-g_psi}. Taking into account that $(g_n)$ converges to~$\id_\UH$, we get
$$
 c_n=\beta_n\big(g_n(i)\big)-\beta_n(i)=\beta^*_n(g_n(i)) ~\,\longrightarrow\,~ \beta_*(i)=0,\quad \text{as~$~n\to+\infty$.}
$$
Moreover, passing to a further subsequence, we may suppose that
$ \arg c_n\to \theta$ as ${n\to+\infty}$
for some $\theta\in\Real$. Finally, removing a finite number of terms, in what follows we assume that $\beta_n$ is univalent in~$\Pi_{1/2}$ for all~${n\in\Natural}$.

The strategy of the remaining part of Step~1 is as follows. Given ${t\ge0}$, we choose a map $k:\Natural\to\Natural$ such that ${k(n)c_n\to te^{i\theta}}$ as ${n\to+\infty}$. Passing if necessary to a further subsequence, we will show that $(g_n^{\circ k(n)})$ converges to some holomorphic self-map ${q_t^*:\UH\to\UH}$. Moreover, our argument will reveal that $q^*_t$ depends neither on the choice of the map ${k:\Natural\to\Natural}$, nor on the subsequence we pass to. In fact, $q^*_t={\beta_*^{-1}\circ(\beta_*+te^{i\theta})}$. We will also see that $q^*_t$'s commute with ${z\mapsto z+1}$ and that ${q^*_t(\Omega)\subset\Omega}{:=h_\varphi(\UD)}$. This would allow us to construct a non-trivial continuous one-parameter semigroup $(\phi_t)$ contained in $\Zn(\varphi)\cap{\overline{\{\psi_n^{\circ k}:n,k\in\Natural\}}}$ by setting
   $\phi_t:={h_\varphi^{-1}\circ q^*_t\circ h_\varphi}$ for all~${t\ge0}$.

\medskip

   Following the programme described above, fix an arbitrary ${t\ge0}$ and consider the function $k(n):={\lceil t/|c_n|\rceil}$, ${n\in\Natural}$, where $\lceil x \rceil$ stands for the smallest integer greater or equal to~$x$. Recalling that ${c_n\to0}$, it is not difficult to see that ${k(n)|c_n|\to t}$ as ${n\to+\infty}$. As a consequence, ${k(n)c_n\to te^{i\theta}}$ as ${n\to+\infty}$.
For ${n\in\Natural}$ denote
   $$
      q_n:=g_n^{\circ k(n)}.
   $$
   These functions form a normal family in~$\UH$. Therefore, passing to a subsequence we may assume that $(q_n)$ converges locally uniformly in~$\UH$ to some map~$q_t^*$. Bearing in mind that $q_n$'s commute with ${z\mapsto z+1}$ and using Theorem~\ref{TH_Hurwitz}, we see that $q^*_t$ is either a non-constant holomorphic self-map of~$\UH$ commuting with ${z\mapsto z+1}$, or $q^*_t\equiv\infty$.

The latter alternative is impossible. Indeed,
       ${\beta_n\big(q_n(w)\big)-\beta_n(w)}=k(n)c_n$
for all~$n\in\Natural$ and any ${w\in\UH}$. Since $\lim_{n\to+\infty}k(n)|c_n|=t$, we have
    $$
     b:=\tfrac32+4\sup\nolimits_{n\in\Natural}k(n)|c_n|~<~+\infty.
    $$
    Now fix some $n\in\Natural$ and set $w_0:=ib$, $w_1:=q_n(w_0)$.
    Note that $\Im w_1\ge\Im w_0\ge 3/2$ and hence, $w_0,w_1\in\Pi_{1/2}$. Using the inequality  $\rho^*_{1/2}(w_0,w_0+1)<(2\Im w_0-1)^{-1}$, we can
    apply Lemma~\ref{LM_Growth} for the function ${\beta:=\beta_n}$ in $\Pi_{1/2}$ to show that $q_n(w_0)$ stays within pseudo-hyperbolic distance~$1/2$  from~$w_0$. Thus, the limit function of $(q_n)$ has to be a self-map of~$\UH$.
In particular, using this fact, we can pass in
${\beta_n\circ q_n}={\beta_n+k(n)c_n}$ to the limit as ${n\to+\infty}$.
This yields ${\beta_*\circ q_t^*=\beta_*+te^{i\theta}}.$

Recall that ${g_n\circ h_\varphi}={h_\varphi\circ\psi_n}$, see again Remark~\ref{RM_beta-g_psi}. As a consequence, ${g_n(\Omega)\subset\Omega}:=h_\varphi(\UD)$ and hence also ${q_n(\Omega)\subset\Omega}$ for all~${n\in\Natural}$. Bearing in mind that the limit map~$q_t^*$ is not constant, one can appeal again to Theorem~\ref{TH_Hurwitz} to show that ${q_t^*(\Omega)\subset\Omega}$. This allows us to construct a holomorphic self-map ${\phi_t:\UD\to\UD}$ by setting $\phi_t:={h_\varphi^{-1}\circ q_t^*\circ h_\varphi}$.
    Then
    \begin{equation}\label{EQ_phi_t-formula}
       \phi_t=(\beta_*\circ h_\varphi)^{-1}\circ(\beta_*\circ h_\varphi+te^{i\theta}).
    \end{equation}
Moreover, we can pass to the limit as~${n\to+\infty}$ in the equality
$$
   \psi_n^{\circ k(n)}=h_\varphi^{-1}\circ q_n\circ h_\varphi
$$
to see that $\psi_n^{\circ k(n)}\to \phi_t$ as $n\to+\infty$.  This shows that $\phi_t$ is contained in the closure of~$\{\psi_n^{\circ k}:n,k\in\Natural\}$. In view of Lemma~\ref{LM_closed}, it further follows that ${\phi_t\in\Zn(\varphi)}$.

It remains to observe that the functions~$\phi_t$ constructed above, as it can be easily checked using formula~\eqref{EQ_phi_t-formula}, form a non-trivial continuous one-parameter semigroup~$(\phi_t)$ in~$\UD$.
This concludes the proof in case $\varphi$ is univalent.

\bigskip	
\noindent\textsf{STEP 2}. Now we provide a proof for the general case. So suppose that $\varphi\in\Hol(\UD)$ is parabolic of positive hyperbolic step, but not necessarily univalent, and that there exists a sequence $(\psi_n)\subset{\Zn(\varphi)\setminus\{\id_\D\}}$ converging locally uniformly in~$\D$ to~$\id_\D$. Consider the self-map $\widetilde\varphi$ and the unital semigroup homomorphism ${\mathcal R_{\varphi}:\Zn(\varphi)\to\Zn(\widetilde\varphi)}$ defined by~\eqref{EQ_varphi-tilde} and~\eqref{EQ_the-tilde-map}, respectively. Recall that $\widetilde\varphi$ is univalent in~$\UD$. Moreover, by \cite[Proposition~2.11\+(2)]{CDG-zero-h.step}, $\widetilde\varphi$ is parabolic of positive hyperbolic step. Hence, taking into account Lemma~\ref{Thm:closedtilde},  the result proved in Step~1 can be applied to the self-map~$\widetilde\varphi$ and the sequence ${(\widetilde\psi_n)\subset\Zn(\widetilde\varphi)\setminus\{\id_\UD\}}$, where ${\widetilde\psi_n:=\mathcal R_\varphi(\psi_n)}$ for all ${n\in\Natural}$.

This shows that the relative closure of $${\{\widetilde\psi_n^{\circ k}:n,k\in\N\}}={\mathcal R_\varphi\big(\{\psi_n^{\circ k}:n,k\in\N\}\big)}$$ in~$\Zn(\widetilde\varphi)$ contains a non-trivial continuous one-parameter semigroup~$(\widetilde\phi_t)$. Lemma~\ref{Thm:closedtilde} implies that $\mathcal R_\varphi\big(\Zn(\varphi)\big)$ is relatively closed in~$\Zn(\widetilde\varphi)$. Therefore, ${(\widetilde\phi_t)\subset\mathcal R_\varphi\big(\Zn(\varphi)\big)}$. Moreover, again by Lemma~\ref{Thm:closedtilde}, the map $\mathcal R_\varphi$ is injective and closed. Therefore, it admits a continuous left inverse
$
 {\mathcal L_\varphi: \mathcal R_\varphi\big(\Zn(\varphi)\big) \to \Zn(\varphi),}
$
which is injective and continuous. Moreover, it is easy to see that $\mathcal L_\varphi$ is a unital semigroup homomorphism. It follows that the family $(\phi_t)_{t\ge0}$ defined by ${\phi_t:=\mathcal L_\varphi(\widetilde\phi_t)}$ for all ${t\ge0}$ is a non-trivial continuous one-parameter semigroup contained in~$\Zn(\varphi)$. It remains to mention that
the univalence of~$\varphi$ follows now by~\cite[Proposition~3.3]{CDG-Embedding}.
\end{proof}

\section{Embeddability of self-maps symmetric w.r.t. $\Real$}\label{S_real}
As we mentioned in the introduction, existence of a non-trivial continuous one-parameter semigroup contained in the centralizer $\Zn(\varphi)$ of a non-elliptic self-map ${\varphi\in\Hol(\UD)}$ is necessary and sufficient for embeddability only if $\varphi$ is hyperbolic, and it fails to be sufficient for embbedability of a parabolic self-map~$\varphi$.

Corollary~\ref{Cor:embedding} asserts that the presence of a boundary fixed point other than the Denjoy\,--\,Wolff point serves as an additional condition ensuring embeddability in the parabolic case. The reason for this phenomenon is that according to \cite[Theorem~1.4]{CDG-Embedding}, a non-elliptic self-map (different from a hyperbolic automorphism and) satisfying this addition condition commutes with a non-trivial continuous one-parameter semigroup~$(\phi_t)$ if and only if it is an element of~$(\phi_t)$.

In this section we give another simple additional condition that ensures embeddability, although in this case, the one-parameter semigroup constructed in the proof of Theorem~\ref{Thm:embedding} does not have to contain~$\varphi$; see Remark~\ref{RM_many-semigroups} in the end of this section.

\begin{corollary}\label{CR_another-add-condition}
Let $\varphi\in\Hol(\UD)$ be a non-elliptic self-map such that $\id_\UD$ is not isolated in~$\Zn(\varphi)$.
Suppose additionally that
\begin{equation}\label{EQ_another-add-condition}
\varphi\big((-1,1)\big)\subset(-1,1).
\end{equation}
Then $\varphi$ is embeddable.
\end{corollary}

\begin{proof}
If $\varphi$ is hyperbolic, then there is nothing to prove because as we have shown in Section~\ref{SS_hyperbolic}, in such a case the desired conclusion holds even without the additional hypothesis~\eqref{EQ_another-add-condition}.

Therefore, we may assume that $\varphi$ is parabolic. Moreover, thanks to condition~\eqref{EQ_another-add-condition}, we know that the Denjoy\,--\,Wolff point of~$\varphi$ is either ${\tau=1}$ or~${\tau=-1}$ and that the orbit ${\big(\varphi^{\circ n}(0)\big)_{n\in\Natural}}$ of the point~$0$ converges to~$\tau$ non-tangentially. It follows that the parabolic self-map~$\varphi$ is necessarily of zero hyperbolic step, see e.g. \cite[Corollary~4.6.10]{Abate2}. Finally, $\varphi$ is univalent in~$\UD$ by Theorem~\ref{Thm:embedding}.

For a holomorphic function $f\in\Hol(\UD,\C)$ we define
$$
 f^\#(z):=\overline{f(\xoverline[1.05] z)},\quad z\in\UD.
$$
Note that $\varphi^\#$ agrees with~$\varphi$ in~$(-1,1)$ by~\eqref{EQ_another-add-condition}. In view of the Identity Principle, we have ${\varphi^\#=\varphi}$ everywhere in~$\UD$. It follows that
$$
  h_\varphi^\#\circ\varphi \,=\, h_\varphi^\#\circ\varphi^\# \,=\, (h_\varphi\circ\varphi)^\#
                            \,=\, (h_\varphi+1)^\# \,=\, h_\varphi^\# + 1.
$$
Since $\varphi$ is univalent in~$\UD$, we have that $h_\varphi$ and hence also~$h_\varphi^\#$ are univalent in~$\UD$, see e.g. \cite[Lemma~3.5.4\+(iv)]{Abate2}. Moreover, $h_\varphi^\#(0)={\overline{h_\varphi(0)}=0}$. By \cite[Theorem~3.1]{CDP-Abel}, it follows that ${h_\varphi=h_\varphi^\#}$. Hence ${\Omega:=h_\varphi(\UD)}$ satisfies
\begin{equation}\label{EQ_Omega-symm}
  \Omega\,=\,\Omega^\#:=\big\{\xoverline w\,:\,w\in\Omega\big\}.
\end{equation}
Since $\varphi$ is a univalent parabolic self-map of zero hyperbolic step, all elements of $\Zn(\varphi)$ are univalent, see \cite[Theorem~4.5\+(1)]{CDG-zero-h.step}. Taking this into account, by \cite[Remark~5.3]{CDG-Centralizer} we further have
$$
  \Zn(\varphi)=\big\{h_\varphi^{-1}\circ(h_\varphi+c)\,:\,c\in\mathcal A_\varphi\big\},\quad
     \mathcal A_\varphi:=\{c\in\Complex:\Omega+c\subset\Omega\}\subset\C\setminus(-\infty,0).
$$
If $c\in\mathcal A_\varphi$, then taking into account~\eqref{EQ_Omega-symm}, we have
$$
  \Omega+\xoverline[0.75]{c\,} \,=\, (\Omega+c)^\# \,\subset\, \Omega^\# \,=\, \Omega.
$$
Hence, $\mathcal A_\varphi = \mathcal A_\varphi^\#:=\{\xoverline[0.75]{c\,}\,:\,c\in\mathcal A_\varphi\}$.

Observe that $\mathcal A_\varphi$ is a subsemigroup of~$[\C,+]$, i.e. if ${c_1,c_2\in\mathcal A_\varphi}$, then ${c_1+c_2\in\mathcal A_\varphi}$, see \cite[Theorem~3.1\+(A)]{CDG-Centralizer}.
Moreover, by \cite[Theorem~5.2]{CDG-Centralizer}, the map
$$
  \mathfrak T_\varphi:\,\mathcal A_\varphi\ni c \mapsto{h_\varphi^{-1}\circ(h_\varphi+c)\in\Zn(\varphi)}
$$
is an isomorphism between the topological semigroups ${[\mathcal A_\varphi,+]}$ and ${[\Zn(\varphi),\,\circ\,]}$ endowed with the topologies inherited from~$\C$ and from $\Hol(\UD)$, respectively.
By Theorem~\ref{Thm:embedding}, the centralizer~$\Zn(\varphi)$ contains a non-trivial continuous one-parameter semigroup~$(\phi_t)$. The function
$$
  [0,+\infty)\ni t \mapsto \mathfrak T^{-1}_\varphi(\phi_t) \ni \mathcal A_\varphi
$$
is additive and continuous. Therefore, see e.g. \cite{Additive}, there exists ${\mu\in\C}$ such that $\mathfrak T^{-1}_\varphi(\phi_t)=\mu t\in\mathcal A_\varphi$ for all~${t\ge0}$. Since ${\mathcal A_\varphi = \mathcal A_\varphi^\#}$, we have ${\xoverline
\mu t\in\mathcal A_\varphi}$ and hence also
\begin{equation}\label{EQ_c.c.}
  \mu t + \xoverline\mu t= 2t\Re \mu\in\mathcal A_\varphi \quad\text{for any~}~t\ge0.
\end{equation}

 The rest of the proof depends on the sign of~$\Re\mu$.

\StepC{1}{$\Re\mu>0$}
In this case, \eqref{EQ_c.c.} implies that~$[0,+\infty)\subset\mathcal A_\varphi$, which means that ${\Omega+t\subset\Omega}$ for any ${t\ge0}$. This condition is known to be equivalent to the embeddability of~$\varphi$, see e.g. \cite[Remark~2.13]{CDG-Centralizer}.

\StepC{2}{$\Re\mu<0$} This situation is impossible. Indeed, if $\Re\mu<0$, then~\eqref{EQ_c.c.} would imply that ${(-\infty,0]\subset\mathcal A_\varphi}$. However, $\varphi$ is parabolic of zero hyperbolic step and hence, it is not an automorphism of~$\UD$. Therefore,  ${\mathcal A_\varphi\subset\C\setminus(-\infty,0)}$ by~\cite[Theorem~3.1\+(B)]{CDG-Centralizer}.

\StepC{3}{$\Re \mu=0$} Note that $\mu\neq0$ because the semigroup $(\phi_t)$ is non-trivial. What we have proved above shows that $\mathcal A_\varphi$ contains the set $\{\mu t,\xoverline\mu t \colon t\ge0\}$, i.e. the whole imaginary axis~$\mathbb I:={\{iy:y\in\Real\}}$. As a consequence, $\Omega$ is a union of straight lines parallel to~$\mathbb I$. Combining this with the fact that ${\Omega\subset\C}$ is a connected open set different from~$\C$ and satisfying ${\Omega+1}={h_\varphi(\UD)+1}={h_\varphi(\varphi(\UD))\subset\Omega}$, it is not difficult to see that $\Omega$ is a half-plane of the form ${\{w:\Re w>a\}}$, for some ${a\in\Real}$. In particular, ${\Omega+t\subset\Omega}$ for any~${t\ge0}$. Thus, also in this case, we may conclude that $\varphi$ is embeddable.
\end{proof}
\begin{remark}\label{RM_many-semigroups}
It is worth mentioning that in contrast to Corollary~\ref{Cor:embedding}, the hypothesis of Corollary~\ref{CR_another-add-condition} does not imply that any non-trivial continuous one-parameter semigroup  contained in the centralizer~$\Zn(\varphi)$ is a reparametrization of the unique continuous one-parameter semigroup~$(\phi_t)$ satisfying ${\phi_1=\varphi}$. As an example, consider the situation of Case~3 in the above proof, when ${\Omega=\{w:\Re w>a\}}$ for some ${a\in\Real}$. In such a case, ${\Omega+te^{i\theta} \subset\Omega}$ for all ${t\ge0}$ and all~${\theta\in[-\pi/2,\pi/2]}$. Therefore, $\Zn(\varphi)$ contains an infinite family of one-parameter semigroups~$(\phi_t^\theta)$ given by $\phi_t^\theta:=h_\varphi^{-1}\circ(h_\varphi+te^{i\theta})$, ${t\ge0}$, ${\theta\in[-\pi/2,\pi/2]}$.
\end{remark}

\section{Concluding remarks}\label{S_concluding}
Denote by $\mathcal G(\Hol(\UD))$ the class of all infinitesimal generators of (arbitrary) continuous one-parameter semigroups in~$\UD$.

Let $\varphi\in\Hol(\UD)$ be a non-elliptic self-map and let $\mathcal G(\Zn(\varphi))$ stand for the subset of $\mathcal G(\Hol(\UD))$ formed by the infinitesimal generators of the one-parameter semigroups contained in~$\Zn(\varphi)$. According to the Berkson\,--\,Porta representation formula, see e.g. \cite[Theorem~10.1.10]{BCD-Book}, if $G\in\mathcal G(\Hol(\UD))$ does not vanish identically, then it has at most one zero in~$\UD$. If $G$ has no zero in~$\UD$, then it generates a non-elliptic one-parameter semigroup~$(\phi_t)$. For this case, we proved in \cite[Section~3]{CDG-Embedding} that ${(\phi_t)\subset\Zn(\varphi)}$ if and only if
\begin{equation}\label{EQ_G-condition}
  G\circ\varphi=\varphi'G.
\end{equation}
In the other case, i.e. when $G$ has a unique zero at some ${\tau\in\UD}$, the situation is simpler: by Behan's result, see Section~\ref{SS_commuting}, $(\phi_t)\not\subset\Zn(\varphi)$ because $\phi_t$'s are elliptic; in turn, $G$ fails to satisfy~\eqref{EQ_G-condition} because otherwise it would vanish also at ${\varphi(\tau)\neq\tau}$.

Thus,  we have the following characterization of $\mathcal G(\Zn(\varphi))$.

\begin{proposition}\label{PR_G-condition}
 Let $\varphi\in\Hol(\UD)$ be a non-elliptic self-map. Then
 $$
    \mathcal G(\Zn(\varphi))~=~\big\{G\in\mathcal G(\Hol(\UD)):\text{\rm $G$ satisfies equation~\eqref{EQ_G-condition}}\big\}.
 $$
\end{proposition}

Since~$\mathcal G(\Hol(\UD))$ is a convex cone and since it is closed in the topology of~$\Hol(\UD,\C)$, see e.g. \cite[Corollary~10.2.7]{BCD-Book},  Proposition~\ref{PR_G-condition} implies the following statement.
\begin{corollary}\label{CR_convex-cone}
 Let $\varphi\in\Hol(\UD)$ be a non-elliptic self-map. Then $\mathcal G(\Zn(\varphi))$ is a topologically closed convex cone in~$\Hol(\UD,\C)$.
\end{corollary}
This corollary implies a couple of further consequences.     We say that one-parameter semigroups $(\phi_t^1)$ and $(\phi_t^2)$ are \dff{essentially distinct} if ${\{\phi_t^1:t\ge0\}}\neq{\{\phi_t^2:t\ge0\}}$.

\begin{remark}\label{RM_convex-cone}
Suppose $\varphi$ is non-elliptic, but not a hyperbolic automorphism. If $\Zn(\varphi)$ contains two essentially distinct non-trivial continuous one-parameter semigroups, then $\Zn(\varphi)$ contains infinitely many pairwise essentially distinct non-trivial continuous one-parameter semigroups. Indeed, if $G_0$ and $G_1$ are infinitesimal generators of these two one-parameter semigroups, then $G_\lambda:={\lambda G_1+(1-\lambda)G_0}\in\mathcal G(\Zn(\varphi))$ for any $\lambda\in[0,1]$. Since the semigroups generated by $G_0$ and $G_1$ are essentially distinct, $G_1$ is not a positive multiple of~$G_0$, i.e. it is not of the form~$\kappa G_0$, where $\kappa>0$ is a constant. It follows that if ${\mu,\lambda\in[0,1]}$ and ${\mu\neq\lambda}$, then $G_\mu$ is not a positive multiple of~$G_\lambda$ either. As a result, the one-parameter semigroups generated by $G_\mu$ and $G_\lambda$ are essentially distinct.
\end{remark}

\begin{remark}
Suppose now that $\varphi\in\Hol(\UD)$ is a (univalent) parabolic self-map of positive hyperbolic step and let $(\phi_t)$ be a non-trivial continuous one-parameter semigroup contained in~$\Zn(\varphi)$. Combining \cite[Proposition~3.3]{CDG-Embedding} with Theorem~\ref{TH_SL0} and taking into account the uniqueness part in Theorem~\ref{TH_SL-positive-step}, it is not difficult to see that $c_{\varphi,\phi_t}=at$ for all ${t\ge0}$ and some constant ${a\in\Complex\setminus\{0\}}$, and that the simultaneous linearization function~$\beta_{\varphi,\phi_t}$ is the same for all~${t>0}$: in fact, the Koenigs function~$H$ of~$(\phi_t)$ satisfies
\begin{equation}\label{EQ_H}
aH-{\beta_{\varphi,\phi_t}\circ h_\varphi}=\mathrm{const},
\end{equation}
which uniquely determines $\beta_{\varphi,\phi_t}$. Recall that the infinitesimal generator of~$(\phi_t)$ is given by ${G=1/H'}$. Bearing this in mind and taking into account that ${\beta:=\beta_{\varphi,\phi_t}}$ satisfies~\eqref{EQ_beta(d)}, it is not difficult to deduce from~\eqref{EQ_H} that two non-trivial continuous one-parameter semigroups contained in~$\Zn(\varphi)$ have different simultaneous linearization functions if and only if their infinitesimal generators are not constant multiples of each other.
The argument used in Remark~\ref{RM_convex-cone} shows then that if there are at least two distinct simultaneous linearization functions associated to non-trivial continuous one-parameter semigroups contained in~$\Zn(\varphi)$, then there are infinitely many such simultaneous linearization functions.
\end{remark}

\begin{example}\label{EX_example}
Let $h$ be a conformal map of~$\UD$ onto the first quadrant $\{w:0<\arg w<\pi/2\}$. Then the Koenigs map $h_\varphi$ of $\varphi:={h^{-1}\circ(h+1)}$ is given by ${h_\varphi=h-\Re h(0)}$. In particular, $\varphi$ is parabolic of positive hyperbolic step and ${S_\varphi=\UH}$. It is not difficult to see that for any~$\alpha\in[0,\pi/2]$, the formula ${\phi_t^\alpha:=h_\varphi^{-1}\circ(h_\varphi+te^{i\alpha})}$ defines a non-trivial continuous one-parameter semigroup $(\phi_t^\alpha)$ contained in~$\Zn(\varphi)$. The corresponding simultaneous linearization function is  ${\beta_{\varphi,\phi_t^\alpha}=\id_\UH}$, the same for all~$\alpha\in[0,\pi/2]$.

Further, let $\beta_0$ be the conformal map of~$\UH$ onto $D:=\C\setminus\bigcup_{k\in\mathbb Z}\{k+\frac{1}{2}+iy:{y\le \tfrac{\log 2}{\pi}}\}$
given by
$$
 \beta_0(w)~:=~w\,-\,\tfrac{1}{\pi i}\,\log\big(1-e^{2\pi i w}\big) ~=~\tfrac12\,+\,\tfrac{i}{2\pi}\,\log\big(4\sin^2(\pi w)\big), \quad w\in\UH,
 $$
where the branches of the logarithm are chosen in such a way that ${\beta_0(w)-w\to0}$ as ${\Im w\to+\infty}$. Note that ${\beta_0(h(\UD))=D\cap\UHR}$. In particular, ${\beta_0(h(\UD))+it}\subset{\beta_0(h(\UD))}$ for all~${t\ge0}$. It follows that the formula
$$
  \psi_t:=(\beta_0\circ h)^{-1}\circ\big(\beta_0\circ h+it\big)=
   (\beta\circ h_\varphi)^{-1}\circ\big(\beta\circ h_\varphi+it\big),
$$ where
                    $\beta(w):=\beta_0(w+\Re h(0))-\Re h(0)$,
defines a non-trivial continuous one-parameter semigroup~$(\psi_t)$. Moreover, using the identity $\beta(w+1)={\beta(w)+1}$ for all ${w\in\UH}$, it is not difficult to check that ${(\psi_t)\subset\Zn(\varphi)}$.
By construction and by uniqueness of the simultaneous linearization function, see Theorem~\ref{TH_SL-positive-step}, we have ${\beta_{\varphi,
\psi_t}=\beta}$ for all ${t>0}$.

Taking convex combinations of $G_0(z):=1/h_\varphi'(z)$ and $G_1(z):=iG_0(z)/\beta'(h_\varphi(z))$,
which are the infinitesimal generators of $(\phi_t^0)$ and $(\psi_t)$, respectively, one obtains an infinite family of one-parameter semigroups in~$\Zn(\varphi)$ whose simultaneous linearization functions are pairwise distinct.
\end{example}

\begin{remark}\label{RM_anglim}
To conclude, it is worth observing that if $\Zn(\varphi)$ contains a non-trivial continuous one-parameter semigroup, then in view of \cite[Proposition~3.3]{CDG-Embedding}, the argument in \cite[Proof of Theorem~3.1]{pavel} shows that $\varphi$ has non-tangential limit at \textit{every} point on~$\UC$.
\end{remark}

\bibliographystyle{amsplain}

\end{document}